\documentclass[leqno]{amsart}
\usepackage{amssymb}
\usepackage{mathrsfs}
\usepackage{amsmath, amsfonts, vmargin, enumerate}
\usepackage{graphics}

\usepackage{verbatim}
\usepackage{amsthm}
\usepackage{latexsym, bm}
\usepackage{euscript}
\usepackage{dsfont}
\usepackage{yhmath}
\usepackage{color}
\usepackage[backref]{hyperref}

\input xypic
\xyoption {all}

\makeatletter

\newtheorem{thm}{Theorem}[section]

\newtheorem{defi}[thm]{Definition}
\newtheorem{remark}[thm]{Remark}
\newtheorem{example}[thm]{Example}
\newtheorem{pb}[thm]{Problem}
\newtheorem{conj}[thm]{Conjecture}

\numberwithin{equation}{section}

\newcommand{\real}{{\mathbb R}}

\newcommand{\ent}{{\mathbb Z}}

\newcommand{\T}{{\mathbb T}}

\newcommand{\A}{{\mathcal A}}

\newcommand{\B}{{\mathcal B}}

\newcommand{\K}{{\mathcal K}}

\newcommand{\M}{{\mathcal M}}

\renewcommand{\a}{\alpha}

\newcommand{\D}{\Delta}

\renewcommand{\t}{\theta}
\renewcommand{\l}{\lambda}
\newcommand{\s}{\sigma}

\newcommand{\Tr}{\mbox{\rm Tr}}

\newcommand{\8}{\infty}

\newcommand{\pd}{\partial}
\newcommand{\la}{\langle}
\newcommand{\ra}{\rangle}

\newcommand{\wh}{\widehat}

\newcommand{\ri}{{\rm{i}}}

\newcommand{\be}{\begin{eqnarray*}}
	\newcommand{\ee}{\end{eqnarray*}}
\newcommand{\beq}{\begin{equation}}
	\newcommand{\eeq}{\end{equation}}
\newcommand{\beqn}{\begin{equation*}}
	\newcommand{\eeqn}{\end{equation*}}

\begin{document}
\title{Schatten classes on noncommutative tori: Kernel conditions}
	
	\thanks{{\it 2000 Mathematics Subject Classification:} Primary: 46L52, 46L51, 46L87. Secondary: 47L25, 43A99}
	
	\thanks{{\it Key words:} Quantum tori, Schatten classes}
	
	\author{Michael Ruzhansky}
	\address{Department of Mathematics: Analysis, Logic and Discrete Mathematics,
                       Ghent University, Belgium and
                       Queen Mary University of London, United Kingdom}
	\email{Michael.Ruzhansky@ugent.be}
	
	\author{Kai Zeng}
	\address{Department of Mathematics: Analysis, Logic and Discrete Mathematics,
                       Ghent University, Belgium}
	\email{kai.ZENG@ugent.be}

	\date{}
	\maketitle

	\markboth{M. Ruzhansky and K. Zeng}
	{Schatten classes on noncommutative tori: Kernel conditions}

\begin{abstract}
		In this note, we give criteria on noncommutative integral kernels ensuring that integral operators on quantum torus belong to Schatten classes. With the engagement of a noncommutative Schwartz' kernel theorem on the quantum torus, a specific test for Schatten class properties of bounded operators on the quantum torus is established. 
		
	\end{abstract}
	

	\tableofcontents
	

\section{Introduction}
Given a closed smooth manifold, the properties of the integral operators defined on it merit thorough examination. For example, finding sufficient conditions of Schwartz integral kernels to ensure the corresponding integral operators belong to different Schatten classes. Identifying such criteria in various domains is a classical problem that has been extensively scrutinized. Specifically, it is well-established that the regularity and smoothness of the kernel are intrinsically linked to the asymptotic behaviors of the singular values.

In the paper by Delgado and Ruzhansky \cite{DR14}, one presented criteria for Schatten classes of integral operators on compact smooth manifolds, their approach is based on the renowned factorization technique, particularly in the manner elucidated by O’Brien \cite{OB82}. Their results encompass compact Lie groups as a special case. The sufficient conditions for integral kernels $K(x,y)$ to belong to Schatten classes will necessitate regularity of a specific order in either $x$, $y$, or both. Their method's advantage lies in the flexibility afforded to the sets of variables upon which the kernel's regularity is imposed. Before long the authors significantly extended their main results of \cite{DR14} to Euclidean space $\mathbb{R}^n$ by means of the anharmonic oscillators in \cite{DR21}. The main results of \cite{DR21} provide a non-compact counterpart of the results in \cite{DR14}. Determining the anharmonic oscillators in the noncommutative setting is still an open question. Namely, to fill this gap if we consider the noncommutative Euclidean space and also establish the criteria for Schatten classes on it, but this will be done elsewhere.

Motivated by the work of Delgado and Ruzhansky \cite{DR14}, we aim to extend the aforementioned criteria for Schatten classes to the noncommutative setting. The inaugural fully non-commutative singular integral operator theory was given by Gonz\'{a}lez-P\'{e}rez, Junges, and Parcet \cite{GJP21}. Put differently, the very first form of an integral operator acting on a  general von Neumann algebra was articulated by the three authors. A crucial point is to define kernels over $\M \overline{\otimes} (\M)^{\rm op}$ where the second copy is the opposite algebra of the von Neumann algebra $\M$.

As a particular instance of the aforesaid definition, we choose the quantum tori $\T_\t^d$ to work with and establish the criteria for Schatten classes on it.  Quantum tori (also known as noncommutative tori and irrational rotation algebras) are landmark examples in noncommutative geometry and can be viewed as an analog of closed manifolds in the non commutative setting. Many aspects of harmonic analysis have been studied for quantum tori and were taken as fundamental examples in noncommutative geometry and noncommutative analysis, see \cite{CXY12, XXY18, MSX19}. 

In this note, we add some regularity conditions on the noncommutative kernel $k \in L^2(\T_\t^d \otimes (\T_\t^d)^{\rm op})$ by engaging the Bessel potentials on $\T_\t^d$. The regularity of kernel $k$ will give us the Schatten class property of the associated integral operator $T_k$. And we consider a bounded linear operator $T$ on $L^2(\T_\t^d)$, we first exploit a noncommutative Schwartz' kernel theorem for $T$ in section \ref{schwartz}, and then invoke the criteria we obtain in section \ref{schatten} to give a characterization of Schatten classes for $T$. Our approach is adapted from \cite{DR14} and \cite{RT16}. 
\section{Notation}\label{notation section}

\subsection{Noncommutative Tori}\label{nc tori subsection}
Harmonic analysis on noncommutative tori is an established subject. The exposition here closely follows \cite{MSX19} and \cite{XXY18}, and for the sake of brevity, we refer the reader to \cite{MSX19} and \cite{XXY18} for a detailed exposition of the topic and provide here only the definitions relevant to this text.
    
\subsubsection{Basic definitions}
Let $d\ge2$ and let $\theta=(\theta_{kj})$ be a real skew symmetric $d\times d$-matrix. The associated $d$-dimensional noncommutative
torus $\mathcal{A}_{\theta}$ is the universal $C^*$-algebra generated by $d$
unitary operators $U_1, \ldots, U_d$ satisfying the following commutation
relation
\beq \label{eq:CommuRelation}
U_k U_j = e^{2 \pi \mathrm{i} \theta_{k j}} U_j U_k,\quad j,k=1,\ldots, d.
\eeq
We will use standard notation from multiple Fourier series. Let  $U=(U_1,\cdots, U_d)$. For $m=(m_1,\cdots,m_d)\in\ent^d$ we define
$$U^m=U_1^{m_1}\cdots U_d^{m_d}.$$
A polynomial in $U$ is a finite sum
$$ x =\sum_{m \in \mathbb{Z}^d}\alpha_{m} U^{m}\quad \text{with}\quad
\alpha_{m} \in \mathbb{C},$$
that is, $\alpha_{m} =0$ for all but
finitely many indices $m \in \mathbb{Z}^d.$ The involution algebra
$\mathcal{P}_{\theta}$ of all such polynomials is
dense in $\A_{\theta}.$ For any polynomial $x$ as above we define
$$\tau (x) = \alpha_{0},$$
where $0=(0, \cdots, 0)$.
Then, $\tau$ extends to a  faithful  tracial state on $\A_{\theta}$.  Let  $\mathbb{T}^d_{\theta}$ be the $w^*$-closure of $\A_{\theta}$ in  the GNS representation of $\tau$. This is our $d$-dimensional quantum torus. The state $\tau$ extends to a normal faithful tracial state on $\mathbb{T}^d_{\theta}$ that will be denoted again by $\tau$. Recall that the von Neumann algebra $\mathbb{T}^d_{\theta}$ is hyperfinite.

There is an alternative way to understand the noncommutative tori, given $m,n \in \mathbb{Z}^d$, define a map from $\mathbb{Z}^d \times \mathbb{Z}^d$ to $\T$:
$$ \s(m,n)= e^{2\pi \ri  m^t\, \tilde\t  \,n }.$$
Here  $m^t$ denotes the transpose of $m=(m_1,\dots,m_d)$ and  $\tilde{\theta}$ is the following lower-triangular $d\times d$-matrix deduced from the skew symmetric matrix $\theta$:
\beq\label{theta-lower}
\tilde{\theta}= \begin{pmatrix}
	0 & 0 & 0&\dots & 0\\
	\t_{2,1} & 0           & 0 &\dots & 0\\
	\t_{3,1} & \t_{3,2}           & 0 &\dots & 0\\
	\vdots&\vdots&\vdots&\vdots&\vdots\\
	\t_{d-1,1} & \t_{d-1, 2}          & \t_{d-1, 3}           &\dots & 0\\
	\t_{d, 1 } & \t_{d, 2 }           & \t_{d, 3 }           &\dots & 0
\end{pmatrix}.\eeq

\medskip

Define a projective representation of $\mathbb{Z}^d$ on $\ell_2(\mathbb{Z}^d)$ induced by the above $\s$:
$$ \lambda_{\s}(m)g(n)= \s(-n,m)g(m-n).$$

Then, for $m,n \in \ent^d$, we have
\beq\label{twisted}
\l_\s(m)\l_\s(n)=\s(m,n)\l_\s(m+n).
\eeq

Consider the twisted group von Neumann algebra $\mathcal{L}_\s(\ent^d)$ generated by $\{\l_\s(m),\ m
\in \ent^d \}$, and a direct computation shows that
$$ U^m U^n= \s(m,n) U^{m+n}.$$

By Stone-von Neumann theorem, $\mathcal{L}_{\s}(\ent^d)$ is $*$-isomorphic to $\T_\t^d$ by mapping $U^m$ to $\l_\s(m)$ for $m \in \ent^d$.

The tracial state $\tau_\s$ on $\mathcal{L}_{\s}(\ent^d)$ is induced by the cyclic vector $\delta_0$, which is the dirac function taking value at the unital element $(0,...,0)$ of $\ent^d$. Namely, for any $x \in \mathcal{L}_\s(\ent^d)$, we obtain that
$$\tau_{\s}(x)=\la x \delta_0, \delta_0 \ra.$$
One can verify that for $x=\sum_{m \in \ent^d}\a_m\l_\s(m)$ with finite sum, we have $\tau_\s(x)=\a_0$. Which indicates that the two traces $\tau$ and $\tau_\s$ coincide.

Moreover, by the structure of the twisted group algebra over $\ent^d$. We can compute the matricial representation of an element $x \in \T_\t^d$. Namely, with the twisted convolution and involution on $\ell_1(\ent^d)$:
\beq\label{ts-convo}
f *_\s  g(m) =  \sum_{n\in \ent^d}  f(m-n) g(n) \sigma(m-n, n ),
\eeq
\beq\label{ts-invo}
f^{\sharp }   (m) =    \sigma(m, -m   )^*\overline{f(-m)}.
\eeq
We gain that
$$ \l_\s(f)\l_\s(g)=\l_\s(f *_\s g) $$
and
$$ \l_\s(f)g=f *_\s g,$$
for $f,g \in \ell_1(\ent^d)$.

Now, we turn to another important perspective for computing the representation of $x \in \T_\t^d$, the noncommutative Fourier transform.

The $L^p$-spaces for $p \in [1,\infty)$ on $\mathbb{T}^d_{\theta}$ are then defined as the operator $L^p$-spaces on $(L^\infty(\mathbb{T}^d_{\theta}),\tau)$,
\begin{equation*}
	L^p(\mathbb{T}^d_{\theta}) :=  L^p(L^\infty(\mathbb{T}^d_{\theta}),\tau).
\end{equation*}
Any $x\in L^1(\T^d_\theta)$ admits a formal
Fourier series:
$$x \sim \sum_{m \in\mathbb{Z}^d} \wh{x} ( m ) U^{m},$$
where \
$$\wh x( m) = \tau((U^m)^*x),\quad m \in \mathbb{Z}^d$$
are the Fourier coefficients of $x$. The operator $x$ is, of course, uniquely determined by its Fourier series.
For $x \in L^2(\mathbb{T}^d_{\theta})$, standard Hilbert space arguments show that it can be written as an $L_2$-convergent
series:
\begin{equation*}
	x = \sum_{n \in \ent^d} \wh{x}(n)U^n,
\end{equation*}
with 
\beq\label{Plancherel-qt}
\|x\|_2^2 = \sum_{n\in \ent^d} | \wh x(n) |^2 .
\eeq  
We will consider the left multiplication operators $M_x$ by mapping $y \in L_2(\T_\t^d)$ to $xy$ as operators on the Hilbert space $\ell_2(\mathbb{Z}^d)$ by virtue of the Plancherel formula \eqref{Plancherel-qt}. Take a left multiplication $M_x$, where $x \sim \sum_{m \in\mathbb{Z}^d} \wh{x} ( m ) U^{m}$. By \eqref{Plancherel-qt}, any $\eta \in L_2(\T^d_\t) $ corresponds to the element $\{ \wh \eta (n)\}_{n \in \ent^d}  \in  \ell_2 (\ent ^d )$. By \eqref{ts-invo} and \eqref{ts-convo}, a direct computation shows that the Fourier coefficient of $M_x \eta$ at $m \in  \ent^d$ is 
$$\sum_{n\in \ent^d} \s(m-n , n ) \wh{x} ( m-n ) \wh \eta (n).$$
Therefore, we can represent the operator $M_x$ as a matrix indexed by $\ent^d$:
\beq\label{Mx-matrix}
\Big( \s(m-n , n ) \wh{x} ( m-n )\Big)_{m,n \in \ent^d}.
\eeq   
The space $C^\infty(\T_\t^d)$ is defined to be the subset of $C(\T_\t^d)$ with the rapid decay sequence of Fourier coefficients $\{\wh{x}_n \}_{n \in \mathbb{Z}^d}$. There is a canonical Fr\'{e}chet topology on $C^\infty(\T_\t^d)$, the topological dual space $\mathcal{D}^{\prime}(\T_\t^d)$ of $C^\infty(\T_\t^d)$ is called the distribution space on $\T_\t^d$.
\medskip

\subsubsection{Harmonic analysis on quantum tori}\label{calculus definition subsubsection}

Many aspects of harmonic analysis on $\T^d$ carry foward to $\mathbb{T}^d_{\theta}$. In the following, we give the definitions of partial differentiations and Fourier multipliers.

\begin{defi}\label{def-diff}
	The partial differentiation operators $\pd_j$, $j = 1,\cdots,d$ on $\mathbb{T}^d_{\theta}$ are defined as:
	\begin{equation*}
		\pd_j (U^n) = 2\pi \ri n_jU^n,\quad n = (n_1,\ldots,n_d) \in \ent^d.
	\end{equation*}
\end{defi}
Every partial derivation $\pd_j$ can be viewed as a densely defined closed (unbounded)
operator on $L^2(\mathbb{T}^d_{\theta})$, whose adjoint is equal to $-\pd_j$.  Let $\Delta=\partial_1^2+\cdots+\partial_d^2$ be the Laplacian. Then $\Delta = - (\pd_1^* \pd_1 + \cdots +\pd_d^* \pd_d)$, so $-\Delta$ is a positive operator on $L^2(\mathbb{T}^d_{\theta})$ with spectrum equal to $\{4\pi^2 |n|^2  : n \in \ent^d\}$. As in the Euclidean case, we let $D_j  = -\ri \pd_j$, which is then self-adjoint.
Given $n =(n_1,\cdots,n_d)\in \mathbb{N}_0^d$ ($\mathbb{N}_0$ denoting the set of nonnegative integers), the associated partial derivation $D^n$ is defined to be
$D_1^{n_1}\cdots D_d^{n_d}$. The order of $D^n$ is   $|n|_1=n_1+\cdots+ n_d$. 
By duality, the derivations transfer to  $\mathcal{D}'(\mathbb{T}^d_{\theta})$  as well. 

\begin{defi}\label{def-FM}
	Let $g$ be a bounded scalar function on $\ent^d$. For $x \in L_2(\T_\t^d)$, the Fourier multiplier $T_g$ with symbol $g$ is defined
	on $x$ by:
	\begin{equation}\label{def-Fourier-Z}
		T_gx = \sum_{n \in \ent^d} g(n)\wh{x}(n)U^n.
	\end{equation}
\end{defi}
By virtue of the Plancherel identity \eqref{Plancherel-qt}, $T_g$ indeed defines a bounded linear operator on $L^2(\T_\t^d)$ and the above series converges in the $L^2$-sense.
If $g$ is unbounded, we may define $T_g$ on the dense subspace of $L^2(\T_\t^d)$ of those $x$ with finitely many non-zero Fourier coefficients.

More generally, if $g$ is a bounded scalar function on $\real^d$, we use the same notation $T_g$ to denote the Fourier multiplier with the symbol $g|_{\ent^d}$. If $g$ is a bounded scalar function on $\real^d\setminus \{0\}$ (or $\ent^d\setminus \{0\}$), $T_g$ is then a Fourier multiplier on the subspace of $\mathcal{D}'(\T^d_\t)$ of all $x$ such that $\wh x(0)=0$.

Typical examples of Fourier multipliers are quantum analogs of Bessel and Riesz potentials. Let $\a\in\real$. Define $J_\a$ on $\real^d$ and $I_\a$ on $\real^d\setminus\{0\}$ by
$$J_\a(\xi)=(1+|\xi|^2)^{\frac\a2}\;\text{ and }\; I_\a(\xi)=|\xi|^{\a}\,.$$
Their associated Fourier multipliers are the Bessel and Riesz potentials of order $\a$, denoted by $J^a$ and $I^\a$, respectively. By duality,  $J^\a$  is also a Fourier multiplier on $\mathcal{D}'(\T^d_\t)$, and $I^\a$ a Fourier multiplier on the subspace of $\mathcal{D}'(\T^d_\t)$ of all $x$ such that $\wh x(0)=0$. Note that
$$J^\a=(1-(2\pi)^{-2}\D)^{\frac\a2}\;\text{ and }\; I^\a=(-(2\pi)^{-2}\D)^{\frac{\a}{2}}.$$

In the same spirit with \eqref{Mx-matrix}, Fourier multipliers on $\T^d_\t$ may be regarded as diagonal matrices acting on $\ell_2 (\ent ^d )$. For instance, the potentials $I^\a$ and $J^\a$ may be represented as 
\beq\label{potentials-matrix}
{\rm{diag}}\{|m|^\a \}_{m\in \ent^d} \quad \mbox{ and  } \quad {\rm{diag}}\{(1+|m|^2)^{\frac \a  2 }\}_{m\in \ent^d}
\eeq
respectively.

It is worthwhile adverting that since the function $ n \mapsto (1+|n|^2)^{-\frac {\a}{2}} $ belongs to $\ell_{\frac {d}{\a},\infty}$, for any $\a >0$ we have
$$J^{-\a}\in S_{\frac {d}{\a},\infty}.$$
Analogously, when $\a > \frac {d}{p}$, it suggests that
$$J^{-\a} \in S_p .$$
For $\a \in \mathbb{R}$, the potential Sobolev space of order $\a$ is defined to be
\be\label{sobolev}
H_p^{\a}(\T_\t^d)=\{x \in \mathcal{D}^{\prime}(\T_\t^d): J^{\a}x \in L^p(\T_\t^d)\},
\ee 
with the norm 
$$\| x\|_{H_p^{\a}(\T_\t^d)}=\|J^{\a}x\|_p. $$
We now define the Sobolev space with mixed regularity $\a_1, \a_2 \ge0$.
\begin{defi}
Let $k \in L^2(\T_\t^d \otimes \T_\t^d)$ and $\a_1, \a_2 \ge 0$,  we say that $k \in H^{\a_1,\a_2}(\T_\t^d\otimes (\T_\t^d)^{\rm op})$ if and only if $(J^{\a_1} \otimes J^{\a_2})k \in  L^2(\T_\t^d\otimes (\T_\t^d)^{\rm op})$.
\end{defi}
\subsection{Schatten classes}\label{operator notation subsection}
   In this subsection let us briefly collect material concerning Schatten classes; for more details, we refer the reader to \cite{Simon1979,Pisier1998,Pisier2001,LSZ12}.
    Let $H$ be a complex separable Hilbert space, and let $\B(H)$ denote the set of bounded operators on $H$, and let $\K(H)$ denote the ideal of compact operators on $H$. Given $T\in \K(H)$, the sequence of singular values $\mu(T) = \{\mu(k,T)\}_{k=0}^\infty$ is defined as
    \begin{equation*}
        \mu(k,T) = \inf\{\|T-R\|\;:\;\mathrm{rank}(R) \leq k\}.
    \end{equation*}
    Equivalently, $\mu(T)$ is the sequence of eigenvalues of $|T|$ arranged in non-increasing order with multiplicities.
    
    Let $p \in (0,\infty).$ The Schatten class $S_p$ is the set of operators $T$ in $\K(H)$ such that $\mu(T)$ is $p$-summable, i.e. in the sequence space $\ell_p$. If $p \geq 1$ then the $S_p$
    norm is defined as
    \begin{equation*}
        \|T\|_p := \|\mu(T)\|_{\ell_p} = \Big(\sum_{k=0}^\infty \mu(k,T)^p\Big)^{1/p}.
    \end{equation*}
    With this norm, $S_p$ is a Banach space, and an ideal of $\B(H)$. For $p=\8$ we set $S_\8 =  \B(H)$ equipped with the operator norm.

The weak Schatten class $S_{p,\infty}$ is defined in an analogous way, by setting operators $T$ such that $\mu(T)$ is in the weak $L^{p}$ space $\ell_{p, \infty}$ with the quasi norm
\begin{equation*}
       \|T\|_{p,\infty} = \sup_{k\geq 0} (k+1)^{1/p}\mu(k,T) < \infty.
    \end{equation*}
    As with the $S_p$ spaces, $S_{p,\infty}$ is an ideal of $\B(H)$.
    
    
    Another equivalent way of defining the above Schatten classes is to consider the whole algebra $\B(H)$ with the usual trace $\Tr$; then the Schatten class $S_p$ is defined as the noncommutative $L_p$ spaces $L_p(\B(H), \Tr)$, see e.g. \cite{PX2003,Xu2007}.
\medskip
\section{Schatten properties of integral operators on noncommutative tori}
\subsection{Schwartz' kernel theorem }\label{schwratz}
Before delving into the main parts of this note, we clarify some notations.

We will denote invariably the dual brackets between $C^\infty(\T_\t^d)$ and $\mathcal{D}^{\prime}(\T_\t^d)$ by $(\cdot, \cdot)$, and the inner product on $L^2(\T_\t^d)$ by $\la \cdot, \cdot \ra$.

In the subsequent discussion, we will show that for a continuous linear operator 
$$T: C^\infty(\T_\t^d) \to \mathcal{D}^{\prime}(\T_\t^d)$$
there exists a kernel $k \in \mathcal{D}^{\prime}(\T_\t^d) \otimes \mathcal{D}^{\prime}(\T_\t^d)$ such that
$$( Tx, y )=( T_k x,y )=( k, y\otimes x).$$
Consider the space $\mathcal{A}$ of all the separately bilinear functionals $A$ on $C^\infty(\T_\t^d) \times C^\infty(\T_\t^d)$. Any distribution $\omega \in \mathcal{D}^{\prime}(\T^d_\t \otimes \T^d_\t) \simeq \mathcal{D}^{\prime}(\T^{2d}_{\t \oplus \t})$ gives rise to such a functional by specialization to the tensor product of two elements:
\beq \label{kernel}
\Lambda \omega (x,y):= \omega(x\otimes y):= A(x,y).
\eeq
The kernel theorem states that the mapping:
$$\Lambda: \omega \mapsto A $$
is a linear homomorphism between $\mathcal{D}^{\prime}(\T^{2d}_{\t \oplus \t})$ and $\mathcal{A}$.

\begin{thm}\label{schwartz kernel}
	For any continuous linear functional $A$ on $C^\infty(\T_\t^d) \times C^\infty(\T_\t^d)$ there exists precisely one distribution $\omega \in \mathcal{D}^{\prime}(\T^{2d}_{\t \oplus \t})$ such that
	\beq\label{schwartz}
 A(x,y)= ( \omega,x\otimes y)
	\eeq
	holds for all $(x,y)\in C^\infty(\T_\t^d) \times C^\infty(\T_\t^d)$. The mapping $A \mapsto \omega$ in \eqref{schwartz} is a linear homeomorphism.	
\end{thm}
\begin{proof}
	Every element $h \in C^{\infty}(\T^d_\t \otimes \T^d_\t)$ can be written in the form 
	$$h=\sum_{m,n \in \mathbb{Z}^d}b_{m,n}U^m \otimes U^n,$$
	where 
	\beq\label{Frechet} 
	\begin{split}
		b_{m,n}&=\tau \otimes \tau(h(U^m \otimes U^n)^{*})\\
		&=\langle h, U^m \otimes U^n \rangle\\
		&=\langle (J^{-\a_1}\otimes J^{-\a_2})(J^{\a_1}\otimes J^{\a_2})h, U^m \otimes U^n \rangle\\
		&= \langle (J^{\a_1}\otimes J^{\a_2})h, (J^{-\a_1}\otimes J^{-\a_2})(U^m \otimes U^n)\rangle.
	\end{split} 
	\eeq
One note that in the previous identities, the indices $\a_1,\a_2$ are arbitrary, we replace them with $\a_1+s_0, \a_2+s_0$ respectively, where $s_0 \in \mathbb{R}$ is a constant. 
	The Cauchy-Schwartz inequality yields that
\beq\label{distribution} 
\begin{split}
|b_{m,n}|&=| \langle (J^{\a_1+s_0}\otimes J^{\a_2+s_0})h, (J^{-(\a_1+s_0)}\otimes J^{-(\a_2+s_0)})(U^m \otimes U^n)\rangle|\\
&\le \| (J^{\a_1+s_0}\otimes J^{\a_2+s_0})h \|_{L^2(\T_\t^d \otimes \T_\t^d)}\| (J^{-(\a_1+s_0)}\otimes J^{-(\a_2+s_0)})(U^m \otimes U^n)\|_{L^2(\T_\t^d \otimes \T_\t^d)}\\
&=  \| (J^{\a_1+s_0}\otimes J^{\a_2+s_0})h \|_{L^2(\T_\t^d \otimes \T_\t^d)}\| (1+|m|^2)^{-\frac {\a_1+s_0} {2}} (1+|n|^2)^{-\frac {\a_2+s_0} {2}}(U^m \otimes U^n)\|_{L^2(\T_\t^d \otimes \T_\t^d)}\\
& \lesssim \|h\|_{H^{\a_1+s_0, \a_2+s_0}}(1+|m|^2)^{-\frac {\a_1+s_0} {2}} (1+|n|^2)^{-\frac {\a_2+s_0} {2}},
\end{split}
\eeq
and choose $s_0$ to be an integer such that $\sum_{m\in \mathbb{Z}^d}(1+|m|^2)^{-\frac {s_0}{2}}<\infty $.
	
	We can rewrite $ h$ as
	\beq\label{expan}
	\begin{split}
		h=&\sum_{m,n \in \mathbb{Z}^d}(1+|m|^2)^{\frac {\a_1} {2}} (1+|n|^2)^{\frac {\a_2} {2}}b_{m,n}((1+|m|^2)^{-\frac {\a_1} {2}}U^m \otimes (1+|n|^2)^{-\frac {\a_2} {2}}U^n )\\
		&= \sum_{m,n \in \mathbb{Z}^d} a_{m,n} V^m \otimes V^n,
	\end{split}
	\eeq
where $\a_1,\a_2\ge 0$ and 
\beq\label{expansion} 
\begin{split}
a_{m,n} &=(1+|m|^2)^{\frac {\a_1} {2}} (1+|n|^2)^{\frac {\a_2} {2}}b_{m,n},\\
V^m &= (1+|m|^2)^{-\frac {\a_1} {2}}U^m, \ V^n= (1+|n|^2)^{-\frac {\a_2} {2}}U^n .
\end{split}
\eeq
Note that
$$ |a_{m,n}|\lesssim \|h\|_{H^{\a_1+s_0, \a_2+s_0}}(1+|m|^2)^{-\frac {s_0} {2}} (1+|n|^2)^{-\frac {s_0} {2}}$$
and 
$$ \|V^m\| \le 1, \ \|V^n\| \le 1$$
since $\a_1, \a_2 \ge0$.
	
We are in a position to construct $\omega$, define
	$$( \omega,h )= \sum_{m,n \in \mathbb{Z}^d}a_{m,n}A(V^m,V^n),$$
	we conclude from \eqref{distribution} and \eqref{expansion} that 
\be
\begin{split}
	|( \omega, h )| &= | \sum_{m,n \in \mathbb{Z}^d}a_{m,n}A(V^m,V^n)|\\
&\le \sum_{m,n \in \mathbb{Z}^d}|a_{m,n}||A(V^m,V^n)| \\
&\lesssim \sum_{m,n \in \mathbb{Z}^d} \|h\|_{H^{\a_1+s_0, \a_2+s_0}}(1+|m|^2)^{-\frac {s_0} {2}} (1+|n|^2)^{-\frac {s_0} {2}}\|A\|\\
&\le C \|h\|_{H^{\a_1+s_0, \a_2+s_0}},
\end{split}
\ee
where the constant $C$ depends on $A$ and $s_0$.
	Hence $\omega$ is a distribution on $\T_\t^d \otimes \T_\t^d$.
	It is apparently the case that $\omega$ satisfies \eqref{schwartz} and $\omega$ is uniquely determined by $A$ since the set of all the finite sums $\sum_{m,n \in \mathbb{Z}^d}b_{m,n}U^m \otimes U^n$ is dense in $C^{\infty}(\T_\t^d\otimes \T_\t^d)$.
	
	Now we show that the map $\Lambda: \omega \mapsto A$ is a linear homeomorphism. To this end, we need to show that $\Lambda$ and $\Lambda^{-1}$ are both continuous.
	
	The Fr\'{e}chet topologies on $\mathcal{A}$ and $\mathcal{D}^{\prime}(\T^{2d}_{\t \oplus \t})$ can be expressed respectively as the following seminorms:
	\beq 
	\begin{split}
		\rho_{B_{\T_\t^d},B_{\T_\t^d}}(A)&={\rm{sup}}\{|A(x,y)|, \ x,y\in B_{\T_\t^d}\},\\
		\rho_{B_{\T_\t^d\otimes \T_\t^d }}(\omega)&={\rm{sup}}\{|(\omega,h) |, \, h \in B_{\T_\t^d\otimes \T_\t^d }\},	
	\end{split}	
	\eeq
	where $B_{\T_\t^d}$ is a bounded set in $C^{\infty}(\T_\t^d)$ and $B_{\T_\t^d\otimes \T_\t^d }$ is a bounded set in $C^{\infty}(\T_\t^d \otimes \T_\t^d)$.
	
	Let $\rho_{B_{\T_\t^d},B_{\T_\t^d}}$ be an arbitrary seminorm in $\mathcal{A}$. Then
	\beq
	\rho_{B_{\T_\t^d},B_{\T_\t^d}}(\Lambda \omega)=\rho_{B_{\T_\t^d},B_{\T_\t^d}}(A) =\sup_{x,y \in{B_{\T_\t^d} }}|A(x,y)|= \sup_{x,y \in{B_{\T_\t^d} }}|( \omega, x \otimes y) |.
	\eeq
	It is clear that there exists a bounded set $B_{\T_\t^d\otimes \T_\t^d } \subset C^{\infty}(\T_\t^d \otimes \T_\t^d)$ such that $B_{\T_\t^d}\otimes B_{\T_\t^d}	\subset B_{\T_\t^d\otimes \T_\t^d }$. Hence
	$$ \sup_{x,y \in{B_{\T_\t^d} }}|( \omega, x \otimes y) |\le \sup_{h \in B_{\T_\t^d\otimes \T_\t^d }}|( \omega, h )|=\rho_{B_{\T_\t^d\otimes \T_\t^d }}(\omega),$$
	which implies that $\Lambda$ is continuous.
	
	Conversely, as for the expansion of $h$ in \eqref{expan}, since $h$ is bounded in $C^{\infty}(\T_\t^d \otimes \T_\t^d)$, we can restrict $V^m, V^n$ to be bounded in $C^{\infty}(\T_\t^d)$ and $\sum_{m,n \in \mathbb{Z}^d}|a_{m,n}|\le 1$. If $\rho_{B_{\T_\t^d\otimes \T_\t^d }}$ is a seminorm on $C^{\infty}(\T_\t^d \otimes \T_\t^d)$, we find
	$$\rho_{B_{\T_\t^d\otimes \T_\t^d }}(\Lambda^{-1}A)=\rho_{B_{\T_\t^d\otimes \T_\t^d }}(\omega)=\sup_{h \in B_{\T_\t^d\otimes \T_\t^d }}|( \omega, h )|=\sup_{h \in B_{\T_\t^d\otimes \T_\t^d }}|\sum_{m,n \in \mathbb{Z}^d}a_{m,n}A(V^m,V^n) |.$$
	Since $V^m,V^n$ are bounded, we can find a bounded set $B_{\T_\t^d}$ which contains $V^m,V^n$. Thus,
	$$\rho_{B_{\T_\t^d\otimes \T_\t^d }}(\Lambda^{-1}A)\le \sum_{m,n \in \mathbb{Z}^d}a_{m,n} \sup_{h \in B_{\T_\t^d\otimes \T_\t^d }}|A(V^m,V^n)|\le \sup_{x,y \in{B_{\T_\t^d} }}|A(x,y)|=\rho_{B_{\T_\t^d},B_{\T_\t^d}}(A),$$
	which demonstrates that $\Lambda^{-1}$ is continuous.
\end{proof}
\begin{remark}
Leveraging the seminal transference method of Chen, Xu, and Yin in \cite{CXY12}, we can transfer the problem of $\T_\t^d$ to the corresponding ones of the operator-valued function in $L^{\infty}(\T^d, \T_\t^d)$. This transference method will provide an alternative proof of Theorem \ref{schwartz kernel} by utilizing the operator-valued Schwartz' kernel theorem, a subject that has yielded a wealth of results. However, the demonstration of the operator-valued theorem is as burdensome as directly proving it on the quantum tori. Moreover, the approach by transference is deficient in the discussions on the topology of $\mathcal{D}^{\prime}(\T_\t^d)$, rendering it somewhat superficial. 
\end{remark}

By consolidating the conclusions we have proven, we can ascertain that for any continuous linear operator $T: C^{\infty}(\T_\t^d) \to \mathcal{D}^{\prime}(\T_\t^d)$, there exists a Schwartz kernel $k \in \mathcal{D}^{\prime}(\T^{2d}_{\t \oplus \t})$ such that
$$( Tx, y )=( T_k x, y ) \ \text{for} \ x,y \in C^{\infty}(\T_\t^d),$$
where $T_k$ is the operator associated to the kernel $k$ implemented by $( T_k x, y ) =( k, y \otimes x )$ for $x,y \in C^{\infty}(\T_\t^d)$.

Now for any $\t$ and $1 \le p \le \infty$, the space $L^p(\T_\t^d)$ embeds into $\mathcal{D}^{\prime}(\T_\t^d)$: an element $x\in L^p(\T_\t^d)$ derives a continuous functional on $ C^{\infty}(\T_\t^d)$ by taking $(x,y)= \tau(xy)$. A map $T \in B(L^2(\T_\t^d))$ can be viewed as a linear map from $C^{\infty}(\T_\t^d)$ to $\mathcal{D}^{\prime}(\T_\t^d)$, by plugging into Theorem \ref{schwartz kernel}, we obtain that there exists $k \in \mathcal{D}^{\prime}(\T^{2d}_{\t \oplus \t})$ such that
$$( Tx, y )=( k, y\otimes x ).$$

Since $C^{\infty}(\T_\t^d)$ is dense in $L^2(\T_\t^d)$, the above identity holds for $x, y \in L^2(\T_\t^d)$. 

It is worth noting that the trace $\tau$ on $\T_\t^d$ and $(\T_\t^d)^{\rm{op}}$ coincides and the Bessel potential $J^{\a}$ acts in the same way for both spaces. So, we have that $L^2(\T_\t^d)\simeq L^2((\T_\t^d)^{\rm op}) $. Indeed, for the opposite algebra $(\T_\t^d)^{\rm op}$, the opposite algebra structure gives the reversed product law for the elements in $\T_\t^d$, namely, if we denote by $\star$ the product operation in $(\T_\t^d)^{\rm op}$, then for the generating elements $U_1,..., U_d$, we gain that
$$U_k \star U_j =U_jU_k= e^{-2 \pi \mathrm{i} \theta_{k j}} U_k U_j= e^{-2 \pi \mathrm{i} \theta_{k j}} U_j \star U_k ,\quad j,k=1,\ldots, d.  $$
By the universal property of quantum tori, we see that $(\T_\t^d)^{\rm{op}}$ is actually isomorphic to $\T_{-\t}^d$.

If we assume that $k \in L^2(\T_\t^d \otimes (\T_\t^d)^{\rm op})$, by the above clarifications, $ L^2(\T_\t^d \otimes (\T_\t^d)^{\rm op}) \subset \mathcal{D}^{\prime}(\T^{2d}_{\t \oplus \t})$. And
\beq \label{unify}
\begin{split}
(Tx, y ) &= ( T_k x, y )\\
&=( k, y\otimes x )\\
&= \tau \otimes \tau(k(y\otimes x))\\
&= \tau(({\rm id}\otimes \tau)(k(1\otimes x))y),
\end{split}
\eeq  
which unveals that $Tx=T_k x= ({\rm id}\otimes \tau)(k(1\otimes x))$ for $x \in L^2(\T_\t^d)$.

\subsection{Schatten classes on noncommutative tori}\label{schatten}
Now, appealing to \eqref{unify}, if we demand futhermore that $k \in H^{\a_1,\a_2}(\T_\t^d \otimes (\T_\t^d)^{\rm op})$, we can deduce that
$$T \in S_r(L_2(\T_\t^d)) \ \text{for} \
r > \frac {2d} {d+2(\a_1+\a_2)}. $$
The conclusion gives us a characterization of the Schatten class property for $T \in B(L^2(\T_\t^d))$ and is presented in the following theorem.

\begin{thm}\label{schatten}
Let $\T_\t^d$ be the quantum tori of dimension $d$ with $d\ge2$ and let $a_1, \a_2 \ge0$. Suppose $k \in H^{\a_1,\a_2}(\T_\t^d \otimes (\T_\t^d)^{\rm op})$. Then the integral operator $T_k$ on $L^2(\T_\t^d) $ defined by
$$T_k(x)=({\rm id}\otimes \tau)(k(1\otimes x))= ({\rm id}\otimes \tau)((1 \otimes x)k),$$
is in the Schatten class $S_r(L^2(\T_\t^d))$ for
$$r > \frac {2d} {d+2(\a_1+\a_2)}. $$	
\end{thm}
\begin{proof}
	Let $J^{\a}$ denote the Bessel potential of order $\a$ on quantum tori, then $J^{\a}$ can be viewed as a Fourier multiplier on $L^2(\T_\t^d)$ with symbol $(1+|n|^2)^{\a/2}$ for $n \in \mathbb{Z}^d$. And we get that $J^{ {-\a_2}} \in S_{p_2}(L^2(\T_\t^d))$ if 
	$$ \a_2 > \frac {d}{p_2}.$$
	If $x \in L^2(\T_\t^d)$, we can  write $x$ formally as $x=\sum_{n \in \mathbb{Z}^d}\wh{x}(n)U^n$, which yields that
	\be
	\begin{split}
		J^{ {-\a_2}}(x)&=\sum_{n \in \mathbb{Z}^d}(1+|n|^2)^{\frac {-\a_2}{2}}\wh{x}(n)U^n\\
		&= \sum_{n \in \mathbb{Z}^d}(1+|n|^2)^{\frac {-\a_2}{2}}\tau((U^n)^*x) U^n\\
		&= \sum_{n \in \mathbb{Z}^d}(1+|n|^2)^{\frac {-\a_2}{2}}({\rm id}\otimes \tau)((U^n\otimes (U^n)^*)(1\otimes x))\\
		&= ({\rm id}\otimes \tau)( \sum_{n \in \mathbb{Z}^d}(1+|n|^2)^{\frac {-\a_2}{2}}(U^n\otimes (U^n)^*)(1\otimes x)),
	\end{split}
	\ee
which means that $J^{-\a_2}(x)=T_{k_2}(x)$ for
$$k_2=\sum_{n \in \mathbb{Z}^d}(1+|n|^2)^{\frac {-\a_2}{2}}(U^n\otimes (U^n)^*).$$
Thus we have determined the kernel $k_2$ for the operator $J^{-\a_2}$ and $J^{\a_2}T_{k_2}(x)=x$.

We observe that for $x \in L^2(\T_\t^d)$,
\be
\begin{split}
	\tau((U^n)^*J^{\a_1}T_k(x))&=\tau((U^n)^*J^{\a_1}({\rm id}\otimes \tau)(k(1\otimes x)))\\
	&=\tau((U^n)^* (J^{\a_1}\otimes \tau)(k(1\otimes x)))\\
	&= \tau((U^n)^*({\rm id}\otimes \tau)(J^{\a_1}\otimes {\rm id})(k(1\otimes x)))\\
	&= \tau (({\rm id}\otimes \tau)((U^n)^*\otimes 1)(J^{\a_1}\otimes {\rm id})(k(1\otimes x)))\\
	&= \tau \otimes \tau (((U^n)^*\otimes 1)(J^{\a_1}\otimes {\rm id})(k(1\otimes x))\\
	&= \tau \otimes \tau(((U^n)^*\otimes 1)(J^{\a_1}\otimes {\rm id})(k(1\otimes J^{\a_2}T_{k_2}(x))))\\
	&= \tau \otimes \tau (((U^n)^*\otimes 1)(J^{\a_1}\otimes {\rm id})k({\rm id} \otimes J^{\a_2})(1\otimes T_{k_2}(x)))\\
	&= \langle ((U^n)^*\otimes 1)(J^{\a_1}\otimes {\rm id})k, ({\rm id} \otimes J^{\a_2})(1\otimes T_{k_2}(x))^*\rangle\\
	&= \langle ({\rm id} \otimes J^{\a_2})((U^n)^*\otimes 1)(J^{\a_1}\otimes {\rm id})k, (1\otimes T_{k_2}(x))^*\rangle\\
	&= \langle ((U^n)^*\otimes 1)({\rm id} \otimes J^{\a_2})(J^{\a_1}\otimes {\rm id})k, (1\otimes T_{k_2}(x))^*\rangle \\
	&= \tau \otimes \tau((U^n)^*\otimes 1)(J^{\a_1}\otimes J^{\a_2})k(1 \otimes T_{k_2}(x)))\\
	&=\tau((U^n)^*({\rm id}\otimes \tau)((J^{\a_1}\otimes J^{\a_2})k(1 \otimes T_{k_2}(x))))\\
	&= \tau((U^n)^*T_{(J^{\a_1}\otimes J^{\a_2})k}(T_{k_2}x)),
\end{split}
\ee 
where the above inner product $\langle \cdot, \cdot \rangle$ on $L^2(\T_\t^d \otimes (\T_\t^d)^{\rm op})$ is induced by the trace $\tau \otimes \tau$, and the inner product makes sense since we know that $(J^{\a_1}\otimes J^{\a_2})k \in L^2(\T_\t^d \otimes (\T_\t^d)^{\rm op})$, thus $((U^n)^*\otimes 1)(J^{\a_1}\otimes J^{\a_2})k \in L^2(\T_\t^d \otimes (\T_\t^d)^{\rm op}) $. 

Now we set $k_1=(J^{\a_1}\otimes J^{\a_2})k$, we have demonstrated that 
$$J^{\a_1}T_k= T_{k_1}T_{k_2}.$$
By the assumption that $k \in H^{\a_1,\a_2}(\T_\t^d \otimes (\T_\t^d)^{\rm op})$ we know that $k_1 \in L^2(\T_\t^d \otimes (\T_\t^d)^{\rm op})$ and thus $T_{k_1} \in S_2$. With the restriction on $\a_2$, $T_{k_2} \in S_{p_2}$. By the Hölder inequality, we see that,
$$J^{\a_1}T_k \in S_t,$$
with $\frac 1 t = \frac 1 2 + \frac 1 {p_2} $.

On the other hand, since $J^{-\a_1} \in S_{p_1}$ when $\a_1 > \frac d {p_1}$, we deduce that
$$ T_k=J^{-\a_1} J^{\a_1}T_k \in S_r$$
with 
$$\frac 1 r = \frac 1 t + \frac 1 {p_1}= \frac 1 2 + \frac 1 {p_1}+ \frac 1 {p_2}.$$
Engaging the inequalities
$$\a_1 > \frac d {p_1} \ \text{and} \ \a_2 > \frac d {p_2},$$
this is equivalent to saying that
$$r > \frac {2d} {d+2(\a_1+\a_2)} .$$
Now, we conclude the proof by discussing the critical cases:
\item(1). If $\a_1=0$ and $\a_2 >0$, just by eliminating the operator $J^{\a_1}$ in the above reasoning, the rest proceed in the same scheme, 
$$ T_k \in S_t,$$
with $\frac 1 t = \frac 1 2 + \frac 1 {p_2} $ and $\a_2 > \frac d {p_2}$, we deduce that
$$t > \frac {2d} {d+2\a_2} .$$
\item(2). If $\a_2=0$ and $\a_1 >0$, we proceed by taking the adjoint of the operator $T_k$, and by \cite[Lemma 2.2]{GJP21}, $T_k^*= T_{{\rm filp}(k)^*}$ with ${\rm flip}(a\otimes b)=b \otimes a$. Thus, the positions of $\a_1$ and $\a_2$ can be exchanged in the precedent proof, repeating mutatis mutandi, we reach that $T_k^*=T_{{\rm filp}(k)^*} \in S_t$, with $\frac 1 t = \frac 1 2 + \frac 1 {p_1} $ and $\a_1 > \frac d {p_1}$, this is a consequence of case (1) and by engaging the fact that $\|T_k\|_{S_t}=\|T_k^*\|_{S_t}$.
\item(3). If $\a_1=\a_2=0$, then the situation is trivial since $k \in L^2(\T_\t^d \otimes (\T_\t^d)^{\rm op})$, it follows that
$$T_k \in S_2 \subset S_r, \ r>2.$$ 
Now the proof is complete.
\end{proof}

\medskip
\section{Acknowledgements}
The authors were supported by the FWO Odysseus 1 grant G.0H94.18N: Analysis and Partial Differential Equations, the Methusalem programme of the Ghent University Special Research Fund (BOF) (Grant number 01M01021). Michael Ruzhansky is also supported by EPSRC grant EP/V005529/1 and FWO Senior Research Grant G022821N.

\end{document}